\documentclass[12pt]{amsart}

\usepackage{a4,amsmath,amssymb,url}
\usepackage{enumitem}

\usepackage{tikz}

\theoremstyle{plain}
\newtheorem{thm}{Theorem}[section]
\newtheorem{lem}[thm]{Lemma}
\newtheorem{cor}[thm]{Corollary}

\newtheorem{prop}[thm]{Proposition}

\theoremstyle{definition}

\newtheorem{prob}[thm]{Problem}

\newtheorem{exa}[thm]{Example}

\theoremstyle{remark}
\newtheorem{rem}[thm]{Remark}

\newcommand{\A}{A}
\newcommand{\B}{B}
\newcommand{\C}{C}
\newcommand{\D}{D}
\newcommand{\F}{F}
\newcommand{\G}{G}
\renewcommand{\H}{H}
\newcommand{\I}{I}
\newcommand{\J}{J}
\newcommand{\K}{K}
\renewcommand{\L}{L}
\newcommand{\M}{M}

\newcommand{\R}{R}
\renewcommand{\S}{S}
\newcommand{\T}{T}

\newcommand{\ZZ}{Z}

\newcommand{\N}{{\mathbb{N}}}
\newcommand{\Z}{{\mathbb{Z}}}

\newcommand{\vV}{{\mathcal{V}}}

\newcommand{\m}{\mathrm{cube}}

\newcommand{\setsuchthat}{\: : \: }

\newcommand{\leqsd}{\leq_{\textup{sd}}}

\newcommand{\join}{\vee}

\date{\today}
\thanks{The first author was supported by the National Science Foundation under Grant No. DMS 1500254.}

\begin{document}


\title{Generating subdirect products}

\author{Peter Mayr}
\address[Peter Mayr]{Department of Mathematics, CU Boulder, USA}
\email{peter.mayr@colorado.edu}
\author{Nik Ru\v{s}kuc}
\address[Nik Ru{\v{s}}kuc]{School of Mathematics and Statistics, University of St Andrews, St Andrews, Scotland, UK}
\email{nik.ruskuc@st-andrews.ac.uk}

\begin{abstract}
We study conditions under which subdirect products of various types of algebraic structures are finitely generated or finitely presented.
In the case of two factors, we prove general results for arbitrary congruence permutable varieties, which generalise previously known results for groups, and which apply to modules, rings, $K$-algebras and loops.
For instance, if $C$ is a fiber product of $A$ and $B$ over a common quotient $D$, and if $A$, $B$ and $D$ are finitely presented, then $C$ is finitely generated.
For subdirect products of more than two factors 
we establish a general connection with projections on pairs of factors and higher commutators.
More detailed results are provided for groups, loops, rings and $K$-algebras.
In particular, let $C$ be a subdirect product of $K$-algebras $A_1,\dots,A_n$ for a Noetherian ring $K$ such that
the projection of $C$ onto any $A_i\times A_j$ has finite co-rank in $A_i\times A_j$.
Then $C$ is finitely generated (resp. finitely presented) if and only if all $A_i$ are finitely generated
(resp. finitely presented).
Finally, examples of semigroups and lattices are provided which indicate further complications as one ventures beyond congruence permutable varieties.
\end{abstract}

\maketitle

\section{Introduction}   \label{sec:intro}

 A \emph{subdirect product} of algebraic structures $\A$ and $\B$ is a subalgebra $\C$ of the direct product
 $\A\times\B$ such that the projections 
of $\A\times \B$ onto $\A$ and $\B$ both remain surjective when restricted to $\C$.
 Then we write $\C\leqsd\A\times\B$.
 Although the notion of subdirect product is superficially similar to that of direct product, there are significant
 differences between the two. First note that a subdirect product is not uniquely determined by its factors.
 For example, for any  $\A$ both the direct product $\A\times\A$ and the diagonal subgroup given by
 $\{(a,a)\setsuchthat a\in A \}$ are subdirect products.
 Conceptually, the subdirect product is more akin to a \emph{property} than to an actual \emph{construction}.
 One way to \emph{construct} subdirect products are \emph{fiber products} or \emph{pullbacks}
 (see Section~\ref{sec:prelim}).

 In~\cite{MR:FPD} we investigated how properties such as being finitely generated, finitely presented, residually finite
 behave under direct products.
 There we have seen that the properties of $\A\times\B$ closely follow those of the factors $\A$ and $\B$, and
 that it is relatively hard to find counter-examples to the `expected' preservation result:
\begin{equation} \label{eq:direct}
 \A\times\B \text{ has property } \mathcal{P} \text{ if and only if } \A \text{ and } \B \text{ have property } \mathcal{P}.
\end{equation}
In fact, when $\mathcal{P}$ is finite generation, \eqref{eq:direct} holds in all congruence permutable varieties and all varieties with idempotent operations 
\cite[Theorems 2.2, 2.5]{MR:FPD}, although not for semigroups \cite{RRW:GRDP}.
When $\mathcal{P}$ is finite presentability, however, the situation is more complicated.
There exist
direct products of finitely presented loops or lattices
 that are finitely generated but not finitely presented  \cite[Theorems 3.8, 3.10]{MR:FPD}.

 The freedom  provided by subdirect products further complicates matters.
 For instance it is well known that there
 are subdirect products of two free groups of rank $2$ that are:
\begin{itemize}[topsep=0mm,partopsep=0mm,itemsep=1mm,leftmargin=7mm]
\item
not finitely generated \cite[Example 3]{BM:SFP};
\item
finitely generated but not finitely presented \cite{Gr:SGWC};
\item
 finitely generated but with an undecidable membership problem 
  \cite{Mi:OPDP}.
\end{itemize}

 This suggests that for subdirect products we are unlikely to find general results along the lines of~\eqref{eq:direct},
 and that a systematic study of combinatorial properties of subdirect products will
 be much more complicated than is the case of direct products. 
Following some early work such as \cite{BBMS:FPNPC}, Bridson and Miller explicitly
 initiate such a systematic study for groups in \cite{BM:SFP}, and a number of papers by various authors follow; see \cite{BHM:FPSP} for a major recent article.
The majority of results arising
 from this program of research have a specific group theoretic flavour, to do with the geometrical or topological
 machinery.
 However, they are all built on 
 a small number of results which appear to have a more general algebraic flavour, and potential to be valid in other (different or more general) settings.

Taking this observation as our point of departure, 
in the current paper we start an investigation of finite generation of subdirect products in different classes of algebraic structures, and at different levels of generality.
For those classical algebras (modules, rings) for which finite presentability is preserved under
 direct products, we also study finite presentations of subdirect products.

The paper is structured as follows. In Section \ref{sec:prelim} we review the necessary preliminaries regarding algebras, subdirect products, fiber products, as well as commutators and nilpotence in a general algebraic setting. Then in Section \ref{secFGPK} we prove some general results concerning subdirect products of two factors in congruence permutable varieties. They generalise known results from groups established in \cite{BM:SFP}, and apply to modules, rings, $K$-algebras and loops.
In Section \ref{sec:mult} we move onto considerations of several factors in the general setting, followed by Sections \ref{sec:group} and \ref{sec:Kalgebras} where we establish more detailed results for groups, loops,  modules, rings and $K$-algebras.
Finally, in Section \ref{sec:smgplat} we present some examples of subdirect products of monoids and lattices which show that our general results for congruence permutable varieties do not carry over to these structures.

\section{Preliminaries}  \label{sec:prelim}

\subsection{Algebras}

 We refer to~\cite{BS:CUA} for basic notions of universal algebra.
 For an algebra $\A$ with subalgebra $\B$ we write $\B\leq\A$. For $X\subseteq A$, we let $\langle X\rangle$
 denote the subalgebra of $\A$ generated by $X$. An algebra $\A$ is \emph{finitely generated} if there exists
 finite $X\subseteq A$ with $\A = \langle X\rangle$.

 Recall that a \emph{variety} is a class of algebras of the same type that is defined by equations.
 A variety $\vV$ is \emph{congruence permutable} if for every algebra $\A$ in $\vV$ and all congruences $\alpha,\beta$
 of $\A$ we have that $\alpha\circ\beta = \beta\circ\alpha$. Equivalently, there exists a \emph{Mal'cev term}
 $m$ in the language of $\vV$ such that $m(x,x,y) = m(y,x,x) = y$. Hence algebras in congruence permutable varieties
 are also called \emph{Mal'cev algebras}. Note that groups, rings, modules, $\K$-algebras, loops, etc.,
 are Mal'cev algebras.

 A variety $\vV$ is \emph{congruence distributive} (\emph{congruence modular}) if the lattice of congruences for
 each algebra $\A$ in $\vV$ is \emph{distributive} (\emph{modular}). Lattices form a congruence distributive
 variety. Both congruence permutable and congruence distributive imply congruence modular. In general semigroups
 and monoids have none of these properties.

\subsection{Subdirect and fiber products}

 For $\C\leqsd\A\times\B$ let $\pi_A\colon C\to A$, $\pi_B\colon C\to B$ denote the projections onto $A,B$,
 respectively. The projection kernels $\ker\pi_A$ and $\ker\pi_B$ meet in the trivial congruence $0_\C$ of $\C$.
 Further $\C = \A\times\B$ if and only if the projection kernels permute and their join equals the total congruence $1_\C$.
 See Figure~\ref{fig:consd} for part of the congruence lattice of $\C$.

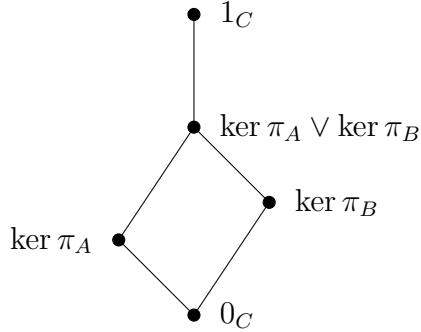
\begin{figure}
\begin{center}
\begin{tikzpicture}
\draw[fill] (2,0) circle[radius=0.8mm]; \node [right] at (2.2,0) {$0_\C$};
\draw (2,0)  -- (1,1);
\draw[fill] (1,1) circle[radius=0.8mm]; \node [left] at (0.8,1) {$\ker\pi_\A$};
\draw (2,0)  -- (3,1.5);
\draw[fill] (3,1.5) circle[radius=0.8mm]; \node [right] at (3.2,1.5) {$\ker\pi_\B$};
\draw (1,1) -- (2,2.5);
\draw (3,1.5) -- (2,2.5);
\draw[fill] (2,2.5) circle[radius=0.8mm]; \node [right] at (2.2,2.5) {$\ker\pi_\A\vee\ker\pi_\B$};
\draw (2,2.5) -- (2,4);
\draw[fill] (2,4) circle[radius=0.8mm]; \node [right] at (2.2,4) {$1_\C$};
\end{tikzpicture}
\caption{Some congruences of $\C\leqsd\A\times\B$.}
\label{fig:consd}
\end{center}
\end{figure}

 Note that $\C/\ker\pi_\A\vee\ker\pi_\B$ is a common homomorphic image of $\A$ and of $\B$.
 By the Homomorphism Theorem the relations
\begin{align*}
&\lambda_\A:=\pi_A(\ker\pi_A\vee\ker\pi_B),\\
&\lambda_\B:=\pi_B(\ker\pi_A\vee\ker\pi_B)
\end{align*}
 are congruences on $\A$ and $\B$, respectively, and 
\[ \A/\lambda_\A\cong\B/\lambda_\B\cong\C/\ker\pi_\A\vee\ker\pi_\B. \]
  We call $\C/\ker\pi_\A\vee\ker\pi_\B$ the \emph{quotient} of the subdirect product, and the congruences 
 $\lambda_\A,\lambda_\B$ on the factors $\A,\B$, respectively, \emph{factor kernels}.

 The above observations lead to the main construction method for subdirect products which we now introduce.
 Let $\A$ and $\B$ be algebras with epimorphisms $g\colon\A\rightarrow \D$ and $h\colon\B\rightarrow \D$
 onto a common homomorphic image $\D$. Then the subalgebra $\C$ of $\A\times\B$ with carrier set
\[ C := \{ (a,b)\in A\times B\setsuchthat g(a)=h(b)\} \]
 is called a \emph{fiber product} (or \emph{pullback}) of $\A$ and $\B$ 
with respect to $g,h$.
 Note that $C$ is a subdirect product with quotient $D$ and factor kernels $\ker f$, $\ker g$.

 As an example, note that every congruence $\alpha$ of an algebra $\A$ is a fiber product $\alpha\leqsd\A\times\A$ with quotient $A/\alpha$ and both factor kernels equal to $\alpha$.

 We record a few straightforward properties of this construction without proof.

\begin{lem} \label{lem:fibered}
 Let $\C$ be a fiber product of $\A$ and $\B$ with respect to $g\colon\A\to\D$, $h\colon\B\to\D$. Then
\begin{enumerate}[label=\textup{(\alph*)}, widest=(b), leftmargin=10mm]
\item $\C\leqsd\A\times\B$;
\item $\lambda_\A = \ker g = \bigl\{ (a_1,a_2)\in A\times A\setsuchthat 
\bigl(\exists b\in B\bigr)\bigl( (a_1,b),(a_2,b)\in C\bigr)\bigr\}$ and
 $\lambda_\B= \ker h = \bigl\{ (b_1,b_2)\in B\times B\setsuchthat 
\bigl(\exists a\in A\bigr)\bigl( (a,b_1),(a,b_2)\in C\bigr)\bigr\}$;
\item \label{it:Clalb}
 $C$ is a union of $\lambda_\A\times\lambda_\B$-classes in $\A\times\B$.
\end{enumerate}
\end{lem}

 A straightforward generalization of Fleischer's Lemma from congruence permutable varieties yields the
 following characterization of those subdirect products that can be obtained from the fiber construction.

\begin{lem}\cite[cf. Fleischer's Lemma, IV.10.1]{BS:CUA} \label{lem:fleischer}
 Let $\C\leqsd\A\times\B$. Then $\C$ is a fiber product if and only if the kernels of the projections $\pi_A$ and $\pi_B$ of $\C$ onto $\A$
 and $\B$ respectively commute.
\end{lem}

\begin{proof}
($\Rightarrow$)
Suppose that $\C$ is the fiber product of $\A$ and $\B$ with respect to
$f:\A\rightarrow\D$, $g:\B\rightarrow\D$, and that
 $(a,b),(c,d)\in C$ are congruent modulo $\ker\pi_A\circ\ker\pi_B$. Then we have $(a,d)\in C$, i.e.,
 $g(a) = h(d)$, so that $g(c) = h(d) = g(a) = h(b)$, which yields $(c,b)\in C$. 
Now
\[ (a,b) \equiv_{\ker\pi_B} (c,b) \equiv_{\ker\pi_A} (c,d) \]
 yields that $(a,b),(c,d)$ are congruent modulo $\ker\pi_B\circ\ker\pi_A$.
 Thus $\ker\pi_A\circ\ker\pi_B\subseteq\ker\pi_B\circ\ker\pi_A$. The converse inclusion follows similarly, and 
 the projection kernels permute.

 ($\Leftarrow$)
The proof is identical to that of \cite[Lemma, IV.10.1]{BS:CUA}, because the only instance where the latter makes use of congruence permutability is to assert that 
the projection kernels commute.
\end{proof}

In particular, in a congruence permutable variety every subdirect product is a fiber product.

\subsection{Multiple factors}

Let $A_1,\dots,A_n$ be algebras of the same type.
For $I\subseteq \{ 1,\dots,n\}$ let
$\pi_I : A_1\times\dots\times A_n\rightarrow \prod_{i\in I} A_i$ denote the natural projection.
We say that a subalgebra $\C\leq \A_1\times\dots\times\A_n$ is a 
\emph{subdirect product} if $\pi_i(C)=A_i$ for all $i\in I$;
we write $\C\leqsd\A_1\times\dots\times\A_n$.
If $\C$ is a subdirect product of $\A_1,\dots,\A_n$ then for each $I\subseteq \{1,\dots,n\}$ the projection $\pi_I(\C)$ is a subdirect product of the $\A_i$ ($i\in I$).
In particular, for any two distinct $i,j\in \{1,\dots,n\}$ we have that $\pi_{ij}(\C)\leqsd \A_i\times \A_j$.

It will transpire in the course of this paper that generally speaking subdirect products of two algebras are more tractable than those with more than two factors.
This is primarily due to the fact that there is no construction playing the role of fiber products in the latter case.
It is therefore natural to attempt to study the general subdirect products by reference to their projections to pairs of components.
In this context we say that $\C$ is \emph{surjective on pairs} if $\pi_{ij}(C)=A_i\times A_j$ for all $i,j\in \{1,\dots,n\}$.
In the more general case where $\pi_{ij}(C)\neq A_i\times A_j$ one considers
$\pi_{ij}(\C)$ as a subdirect product of $\A_i$ and $\A_j$ with the factor kernels given by
$\lambda_{\A_i}=\pi_i (\ker\pi_i\join\ker\pi_j)$
and $\lambda_{\A_j}=\pi_j (\ker\pi_i\join\ker\pi_j)$.
We say that $\C$ is \emph{virtually surjective on pairs} if $\pi_{ij}(C)$ is a subdirect product of $A_i$ and $A_j$
with finite quotient $\A_i/\lambda_{\A_i}\cong\A_j/\lambda_{\A_j}$.
  
 We will also view $\C\leqsd \A_1\times\dots\times\A_n$ as a subdirect product $\C\leqsd \pi_I(\C)\times\pi_J(\C)$,
 where $I$ and $J$ partition $\{1,\dots,n\}$.
In this case the factor kernels are given by
$\pi_I(\ker\pi_I\join\ker\pi_J)$ and
$\pi_J(\ker\pi_I\join\ker\pi_J)$ respectively.

\subsection{Commutators}

 The theory of commutators for normal subgroups of groups has been generalized to congruences of algebras
 in congruence modular varieties (see~\cite{FM:CTC} for binary, \cite{Bu:PCCT,AM:SAHC,Op:RDH,Mo:HCT}
 for higher commutators). We review some notions that will be used later in the paper.

Let $\alpha_1,\dots,\alpha_k$ be congruences of an algebra $\A$ in a congruence modular variety.
 For $i\leq k$ and $a,b\in A$ define
\[ \m^k_i(a,b)\colon \{0,1\}^k\to A,\ x\mapsto 
\left\{\begin{array}{ll} a & \text{if } x_i = 0, \\
          b & \text{if } x_i  = 1. \end{array}\right. \]  
 Let 
\[ M_\A(\alpha_1,\dots,\alpha_k) := \bigl\langle \m^k_i(a,b) \setsuchthat 1\leq i\leq k, (a,b)\in\alpha_i \bigr\rangle \leq\A^{\{0,1\}^k}. \]
Note that $M_\A(\alpha_1,\dots,\alpha_k)$ is a subdirect subalgebra of $\A^{\{0,1\}^k}$.
The $k$-\emph{ary commutator} $[\alpha_1,\dots,\alpha_k]$ is the smallest congruence $\gamma$ of $\A$ such that
 every $f\in M_\A(\alpha_1,\dots,\alpha_k)$ which satisfies 
\[ f(x_1,\dots,x_{k-1},0) \equiv_\gamma f(x_1,\dots,x_{k-1},1)\]
for all $(x_1,\dots,x_{k-1})\in\{0,1\}^{k-1} \setminus\{(1,\dots,1)\}$,
also satisfies
\[ f(1,\dots,1,0) \equiv_\gamma f(1,\dots,1,1). \]
 For $k = 1$ we have $[\alpha_1] = \alpha_1$, for $k=2$ the binary term condition commutator~\cite{FM:CTC},
 and for $k\geq 3$ the higher commutators as introduced by Bulatov~\cite{Bu:PCCT} and developed by
 Aichinger, Mudrinski~\cite{AM:SAHC}, Opr{\v s}al~\cite{Op:RDH}, and Moorhead~\cite{Mo:HCT}.

Commutators for an algebra $\A$ in a congruence modular variety
satisfy the following monotonicity and permutability properties (see \cite{Mo:HCT}):
\begin{enumerate}[label=\textsf{(C\arabic*)}, widest=(C5), leftmargin=10mm]
\item
\label{it:hc1}
$ [\alpha_1,\dots,\alpha_k]  \leq \alpha_1\wedge\dots\wedge\alpha_k$;
\item
\label{it:hc2}
$ [\alpha_1,\dots,\alpha_k]  \leq [\beta_1,\dots,\beta_k] \text{ if } \alpha_1\leq\beta_1,\dots,\alpha_k\leq\beta_k$; 
\item
\label{it:hc3}
 $[\alpha_1,\dots,\alpha_k]  \leq  [\alpha_2,\dots,\alpha_{k}]$; 
\item 
\label{it:hc4}
$ [\alpha_1,\dots,\alpha_k]  = [\alpha_{\sigma(1)},\dots,\alpha_{\sigma(k)}]$ for each  $\sigma\in S_k$;
\item 
\label{it:hc7}
 $\Bigl[\bigvee_{i=1}^\ell\beta_i,\alpha_2,\dots,\alpha_k\Bigr]  = \bigvee_{i=1}^\ell[\beta_i,\alpha_2,\dots,\alpha_k]$;
\item
\label{it:hc8}
$ \bigl[[\alpha_1,\dots,\alpha_k],\beta_1,\dots,\beta_\ell\bigr]  \leq [\alpha_1,\dots,\alpha_k,\beta_1,\dots,\beta_\ell]$.
\end{enumerate}

  For congruences on subdirect products we also note the following:

\begin{lem} \label{le:sdcom}
 Let $\C\leqsd\A\times\B$ with congruences $\gamma_1,\dots,\gamma_k$. Then
\[ [\gamma_1,\dots,\gamma_k] \leqsd \bigl[\pi_\A(\gamma_1),\dots,\pi_\A(\gamma_k)\bigr] \times \bigl[\pi_\B(\gamma_1),\dots,\pi_\B(\gamma_k)\bigr]. \]
\end{lem}

\begin{proof}
Straightforward from
 \[ M_\C(\gamma_1,\dots,\gamma_k) \leqsd M_\A\bigl(\pi_\A(\gamma_1),\dots,\pi_\A(\gamma_k)\bigr) \times M_\B\bigl(\pi_\B(\gamma_1),\dots,\pi_\B(\gamma_k)\bigr). \]
\end{proof}

 We briefly consider two important special cases: groups and rings. 
 For groups, congruences can be identified with normal subgroups, and the binary term condition commutator 
 specializes exactly to the classical binary commutator in groups. By \cite[Lemma 3.6]{Ma:MASC}
 for normal subgroups $N_1,\dots,N_k$
 the $k$-ary commutator is the product over left associated iterated binary commutators for all permutations of
 $N_1,\dots,N_k$,
 \[ [N_1,\dots,N_k] = \prod_{\sigma\in S_k} [\dots[[N_{\sigma(1)},N_{\sigma(2)}],N_{\sigma(3)}],\dots, N_{\sigma(k)}]. \]
For rings (and $\K$-algebras for a commutative ring $\K$ with $1$ -- see Section \ref{sec:Kalgebras} for definition), congruences correspond to ideals.
By \cite[Lemma 3.5]{Ma:MASC} for ideals $I_1,\dots,I_k$
 the $k$-ary commutator is the sum of all products of all permutations of $I_1,\dots,I_k$,
\[ [I_1,\dots,I_k] = \sum_{\sigma\in S_k} I_{\sigma(1)} I_{\sigma(2)} \dots I_{\sigma(k)}. \]

 We note that unlike for groups and rings, for arbitrary Mal'cev algebras $k$-ary commutators cannot be expressed
 as iterated binary commutators, however~\ref{it:hc8} implies
\[ [\dots [[\alpha_1,\alpha_2],\alpha_3],\dots,\alpha_k] \leq [\alpha_1,\alpha_2,\dots,\alpha_k]. \]
  
 This leads to two different notions of nilpotence with the classical one (just called nilpotence) using iterated
 binary commutators to define a central series and the stronger one (supernilpotence) using higher commutators. 
 An algebra $\A$ is \emph{supernilpotent of class} $k$ if the $k+1$-ary commutator of the total congruence $1$
 with itself vanishes but the $k$-ary commutator does not, i.e., $[\underbrace{1,\dots,1}_{k+1}] = 0$ and $[\underbrace{1,\dots,1}_{k}] \neq 0$. Further $\A$ is \emph{virtually supernilpotent of class}
 $k$ if there is a congruence $\alpha$ with $\A/\alpha$ finite such that
 $[\underbrace{\alpha,\dots,\alpha}_{k+1}] = 0, [\underbrace{\alpha,\dots,\alpha}_k] \neq 0$. 
 (Super)nilpotent algebras of class at most $1$ are \emph{abelian}.
Note that for groups and rings nilpotence and supernilpotence coincide, but this is not true in general.

\section{Finite generation for two factors}
\label{secFGPK}

 In \cite{BM:SFP} Bridson and Miller observe the following facts for subdirect products $\C\leqsd\A\times\B$
 of two groups $\A,\B$ 
 with $L_\A := \pi_\A(\C \cap \A\times 1)\unlhd \A$, $L_\B := \pi_\B(\C \cap 1\times\B)\unlhd \B$ 
 the normal subgroups corresponding to the factor kernels: 

\begin{enumerate}[label=\textsf{(G\arabic*)}, widest=(G4), leftmargin=10mm]
\item
\label{SDG1}
 Assume $\C$ is finitely generated. If $\A$ is finitely presented, then $L_\B$ is finitely normally
 generated in $\B$.
The converse holds under the assumption that $\C$ is finitely presented
\cite[Proposition 2.1]{BM:SFP}.
\item
\label{SDG2}
 Assume $\A$ and $\B$ are finitely generated. If $L_\A$ or $L_\B$
 is finitely normally generated, then $\C$ is finitely generated
 \cite[Proposition 2.3]{BM:SFP}.
\item
\label{SDG3}
 Assume $\A$ and $\B$ are finitely presented. Then $\C$ is finitely generated if and only if
 one (and hence both) of $\A/L_\A$ or $\B/L_\B$ is finitely presented
 \cite[Corollary 2.4]{BM:SFP}.
\end{enumerate}

 We 
investigate how far these statements 
generalize to arbitrary algebras. By the following two results, the first part of~\ref{SDG1} holds in complete
 generality, while the second holds for algebras with modular congruence lattices.

\begin{prop} \label{pro:lBfg}
For $\C\leqsd \A\times\B$ the following hold:
\begin{enumerate}[label=\textup{(\alph*)}, widest=(ii), leftmargin=8mm]
\item \label{it:lBfg1}
  If $\ker\pi_\A$ is a finitely generated congruence of $\C$, then $\lambda_\B:=\pi_B(\ker\pi_A\join\ker\pi_B)$
  is a finitely generated congruence of $\B$.
\item \label{it:lBfg2}
 Suppose that the congruence lattice of $\C$ is modular.
 Then $\ker\pi_\A$ is finitely generated if and only if $\lambda_\B$ is finitely generated. 
Moreover, $\ker\pi_\A\join\ker\pi_B$ is finitely generated if and only if both $\ker\pi_A$ and $\ker\pi_B$ are finitely generated.
\end{enumerate}
\end{prop}

\begin{proof}
\ref{it:lBfg1}
If $\ker\pi_A$ is finitely generated, then
$\ker\pi_A\join\ker\pi_B$ is finitely generated modulo $\ker\pi_B$, and hence $\lambda_\B$
is finitely generated in $\pi_B(\C)=\B$.

\ref{it:lBfg2}
By the Isomorphism Theorem for Modular Lattices~\cite[Theorem 348]{Gr:LTF},
the intervals $[\ker\pi_B,\ker\pi_A\vee\ker\pi_B]$ and $[0_C, \ker\pi_A]$ in the congruence lattice of $\C$ are isomorphic.
For the first statement, $\lambda_\B$ being finitely generated is equivalent to the congruence $\ker\pi_A\join\ker\pi_B$ being compact in the interval
$[\ker\pi_B,\ker\pi_A\vee\ker\pi_B]$. This in turn is equivalent to $\ker\pi_A$ being compact
in the interval $[0_C, \ker\pi_A]$, i.e., $\ker\pi_A$ being finitely generated.

For the direct implication of the second statement, suppose that $\ker\pi_\A\join\ker\pi_B$
is finitely generated. Then, in particular, $\ker\pi_\A\join\ker\pi_B$ is finitely generated modulo
$\ker\pi_B$, and hence, as above, $\ker\pi_A$ is finitely generated. By symmetry, $\ker\pi_B$ is finitely generated as well. The converse implication is obvious, and the proof is complete.
\end{proof}

\begin{cor} \label{cor:lBfg}
Let $\A,\B$ belong to a congruence modular variety $\vV$, and suppose $\C\leqsd\A\times\B$ is finitely generated.
If $\A$ is finitely presented in $\vV$, then $\lambda_\B$ is a finitely generated congruence of $\B$.
Moreover, if $\C$ is finitely presented and $\lambda_\B$ is finitely generated, then $\A$ is finitely presented.
\end{cor}

\begin{proof}
The first assertion follows from Proposition \ref{pro:lBfg}\ref{it:lBfg1} because $\C$ finitely generated and $\A$ finitely presented imply that $\ker\pi_\A$ is finitely generated.
The second statement follows from Proposition \ref{pro:lBfg}\ref{it:lBfg2}: 
$\lambda_\B$ finitely generated implies $\ker\pi_\A$ finitely generated, which, together with $\C$ finitely presented, implies that $A\cong\C/\ker\pi_\A$ is finitely presented as well.
\end{proof}

 Proposition~\ref{pro:lBfg} and Corollary~\ref{cor:lBfg} apply in particular to all congruence permutable varieties,
 such as groups, modules, rings and loops, as well as to all congruence distributive varieties, such as lattices.
 In Example~\ref{ex:monlambda} we exhibit a monoid $\C$ with non-modular congruence lattice for which the
 equivalences in Proposition~\ref{pro:lBfg}\ref{it:lBfg2} are not valid.

 The second statement in Corollary~\ref{cor:lBfg}
 raises the question whether there exist finitely presented subdirect products of two non-finitely presented algebras in
 congruence permutable varieties. 
The answer is positive with examples in a variety of modules (Example \ref{exa:Fxymod})
and in the variety of rings (Example \ref{exa:fprings}).
We do not know the answer to this question for groups.

 Statement \ref{SDG2} generalizes from groups to congruence permutable varieties as follows:

\begin{prop} \label{pro:lAfg}
 Let $\A,\B$ be finitely generated algebras in a congruence permutable variety $\vV$, and let $\C\leqsd\A\times\B$.
 If $\lambda_\B$ is a finitely generated congruence on $\B$ (or $\lambda_\A$ is a finitely generated congruence
 on $\A$), then $\C$ is finitely generated.
\end{prop}

\begin{proof}
 Let $X,Y$ be finite generating sets of $\A,\B$, respectively. Let $U$ be a finite generating
 set for $\lambda_\B$. We are going to lift $X,Y,U$ to $C$.
 Since $C$ is subdirect, for every $x\in X$ there exists $b_x\in B$ such that $(x,b_x)\in C$.
 Similarly, for every $y\in Y$ there exists $a_y\in A$ such that $(a_y,y)\in C$.
 For every $(u,v)\in U$ we have $a_{u,v}\in A$ such that $(a_{u,v},u),(a_{u,v},v) \in C$ by Lemma~\ref{lem:fibered}.
 Now let 
\begin{align*}
 X' & = \bigl\{ (x,b_x) \setsuchthat x\in X\bigr\}, \\
 Y' & = \bigl\{ (a_y,y) \setsuchthat y\in Y\bigr\}, \\
 U' & = \bigcup\bigl\{  (a_{u,v},u),(a_{u,v},v) \setsuchthat (u,v)\in U \bigr\}.
\end{align*}
 We claim that 
\begin{equation} \label{eq:CXYU}
 C = \langle X'\cup Y'\cup U'\rangle.
\end{equation}
 To prove this let $(a,b)\in C$. Since $X$ generates $A$, we have some $c\in B$ such that
\[ (a,c)\in\langle X'\rangle. \]
 Note that $(b,c)\in\lambda_\B$. Since $U$ generates $\lambda_\B$,
 we have $(u_1,v_1),\dots,(u_k,v_k)\in U$, $y_1,\dots,y_\ell\in Y$, and a $k+\ell$-ary
 term $t$ such that 
\[ b = t^\B(u_1,\dots,u_k,y_1,\dots,y_\ell),\ c = t^\B(v_1,\dots,v_k,y_1,\dots,y_\ell); \] 
see \cite[Theorem 4.70 (iii)]{MMT:ALV}.
 We lift the arguments of $t^\B$ to $C$ and write
 $d:=t^\A(a_{u_1,v_1},\dots,a_{u_k,v_k},a_{y_1},\dots,a_{y_\ell})$ to obtain
\[ (d,b), (d,c)\in \langle Y'\cup U'\rangle. \]  
 Applying the Mal'cev operation $m$ for the variety $\vV$ to the triple
\[
 (d,b) , (d,c) , (a,c) \in\langle X^\prime\cup Y^\prime\cup U^\prime\rangle
\]
 yields $(a,b)\in\langle X'\cup Y'\cup U'\rangle$. Hence~\eqref{eq:CXYU} is proved.
\end{proof}

 We note that Proposition~\ref{pro:lAfg} does not generalize to congruence modular varieties; for instance
in Example~\ref{exa:lat} we will show that it fails for lattices.
The proposition also fails for monoids; see Example \ref{exmon1}.
 Also the converse of Proposition~\ref{pro:lAfg} does not hold
even for classical congruence permutable varieties,
 as seen in Example \ref{exa:Fxymod} for modules and Example \ref{exa:baumslag} for groups.

 Finally, the statement \ref{SDG3} takes the following form in congruence permutable varieties:

\begin{cor} \label{cor:ABfp}
 Let $\A,\B$ be finitely presented algebras in a congruence permutable variety, and let $\C\leqsd\A\times\B$.
 Then $\C$ is finitely generated if and only if one (and hence both) of $\A/\lambda_\A$ or $\B/\lambda_\B$
 are finitely presented.
\end{cor}

\begin{proof}
 First assume that $\C$ is finitely generated.  Then $\lambda_\B$ is finitely generated by
 Proposition~\ref{pro:lBfg}\ref{it:lBfg1}. Since $\B$ is finitely presented, this implies that $\B/\lambda_\B$
 is finitely presented as well.

 For the converse implication assume that $\B/\lambda_\B$ is finitely presented. Since $\B$ is finitely generated,
 this yields that $\lambda_\B$ is finitely generated. Hence $\C$ is finitely generated by Proposition~\ref{pro:lAfg}.
\end{proof}

 Again, Corollary \ref{cor:ABfp} does not generalize to monoids or lattices; see Examples \ref{exmon1} and \ref{exa:lat}.

\section{Multiple factors in congruence permutable varieties} \label{sec:mult}

 The structure of subdirect products with more than two factors
 turns out to be surprisingly complicated. One problem is that fiber products and Fleischer's
 Lemma~\ref{lem:fleischer} do not behave well under iteration.
It is therefore natural to try to relate finiteness conditions of such products to properties of their projections on pairs.
This is one of the themes in the work of Bridson, Howie, Miller and Short \cite{BHM:FPSP}, where they prove the following results:
\begin{enumerate}[label=\textsf{(G\arabic*)}, widest=(G4), leftmargin=10mm]
\setcounter{enumi}{3}
\item
\label{it:G4}
If $G_1,\dots,G_n$ are finitely presented (resp. finitely generated) groups and $S\leqsd G_1\times\dots\times G_n$ is
virtually surjective on pairs
(which in this case means that the projection of $S$ on any two factors has finite index in their direct product $G_i\times G_j$),
then $S$ is finitely presented (resp. finitely generated) as well.
\end{enumerate}
The finite presentability result is \cite[Theorem A]{BHM:FPSP}; the finite generation result is not explicitly stated or proved, but is implicitly established during the proof of the finite presentability result.
 One may wonder to what extent these carry over to (at least) congruence permutable varieties.

In this vein we prove the following result, showing an intricate connection between subdirect products $\C\leqsd \A_1\times\dots\times\A_n$
 of Mal'cev algebras and their projections to any pair of factors $\A_i\times\A_j$ via (higher)
 commutators of
factor kernels.

\begin{thm} \label{thm:mult}
 Let $n\geq 2$, and let $\C\leqsd \A_1\times\dots\times\A_n$ in a congruence permutable variety.
 For $1\leq i,j \leq n$, let $\lambda_{ij} := \pi_j(\ker\pi_i\vee\ker\pi_j)$
 and $\gamma_j := [\lambda_{1j},\dots,\lambda_{j-1,j},\lambda_{j+1,j},\dots,\lambda_{nj}]$ be congruences on $\A_j$.
 Then
\begin{enumerate}[label=\textup{(\alph*)}, widest=(b), leftmargin=10mm]
\item \label{it:mcom}
 $C$ is a union of $\gamma_1\times\dots\times\gamma_n$-classes in $A_1\times\dots\times A_n$.
\item \label{it:msub}
 If $\D\leq\A_1\times\dots\times\A_n$ such that $\pi_{ij}(\D) = \pi_{ij}(\C)$ for all $1\leq i<j \leq n$ and
 $C/\gamma_1\times\dots\times\gamma_n\subseteq D/\gamma_1\times\dots\times\gamma_n$, then
$\C\leq\D$.
\item \label{it:mfg}
 $\C$ is finitely generated if and only if $\C/\gamma_1\times\dots\times\gamma_n$ and $\pi_{ij}(\C)$ for all $1\leq i<j \leq n$
 are finitely generated. 
\end{enumerate}  
\end{thm}

\begin{proof}
For $n = 2$ item~\ref{it:mcom} of the previous theorem is exactly Lemma~\ref{lem:fibered}\ref{it:Clalb}
 and items~\ref{it:msub}, \ref{it:mfg} are trivial.
So we assume $n\geq 3$ throughout.

 \ref{it:mcom}
For $i\leq n$ and $\rho_i := \ker\pi_i$ we have
\[ \rho_i \leqsd \lambda_{i1}\times\dots\times\lambda_{i,i-1}\times 0\times \lambda_{i,i+1}\times\dots\times\lambda_{in}. \]
 Since any commutator involving $0$ is trivial by~\ref{it:hc1}, 
 Lemma~\ref{le:sdcom} yields
\[ [\rho_1,\dots,\rho_{n-1}] = 0\times\dots\times 0\times [\lambda_{1n},\dots,\lambda_{n-1,n}]. \]
 Since commutators are contained in the intersection of their arguments by~\ref{it:hc1}, this implies
\begin{equation} \label{eq:0gn}
 0^{n-1}\times\gamma_n \leq \bigwedge_{i\in I} \rho_i,
\end{equation}
where $I=\{1,\dots,n-1\}$.
 Now consider $\C$ as a subdirect product of $\pi_I(\C)$ and $\A_n$, and let
 $\delta := \bigl(\bigwedge_{i\in I} \rho_i\bigr)\vee\rho_n$. Lemma~\ref{lem:fibered}\ref{it:Clalb} states that $C$ is a union
 of $\pi_{I}(\delta)\times \pi_n(\delta)$-classes in $\pi_{I}(\C)\times\A_n$.
 Since $\gamma_n \leq \pi_n(\delta)$ by~\eqref{eq:0gn}, this implies that $C$ is a union of
 $0^{n-1}\times\gamma_n$-classes in $A_1\times\dots\times A_n$.
 By symmetry $C$ is a union of $0^{i-1}\times\gamma_i\times 0^{n-i}$-classes for each $i\leq n$, hence also for the
 join of all these congruences.
 Thus~\ref{it:mcom} is proved.

 \ref{it:msub} Note that $\D\leqsd\A_1\times\dots\times\A_n$
 with the same congruences $\lambda_{ij}$ as for $\C$
 by the assumption that $\pi_{ij}(\D)=\pi_{ij}(\C)$ for all
 $1\leq i,j\leq n$. Hence $D$ is also a union of $\gamma_1\times\dots\times\gamma_n$-classes by~\ref{it:mcom}.
 Together with the assumption that $C/\gamma_1\times\dots\times\gamma_n \subseteq D/\gamma_1\times\dots\times\gamma_n$,
 this yields $\C \leq \D$.

 \ref{it:mfg} 
The forward direction is trivial. For the converse assume that
 there exists a finite subset $X$ of $C$ such that $X$ generates $\C$ modulo $\gamma_1\times\dots\times\gamma_n$
 and $\pi_{ij}(X)$ generates $\pi_{ij}(\C)$ for all $i<j$. Now $\D:=\langle X\rangle$  
 satisfies the assumptions of~\ref{it:msub} and consequently $\D = \C$.
 This completes the proof of~\ref{it:mfg} and of the Theorem.
\end{proof}

 Concerning~\ref{it:mfg} of the previous theorem we note that in general for a subdirect product to be finitely
 generated it is not enough that the projection onto any pair of factors is finitely generated; see
 Example~\ref{exa:pairsfg1} for groups.
 In fact, Example~\ref{exa:pairsfg} exhibits for any $n\in\N$ a non-finitely generated $n$-fold subdirect
 product of modules whose projection on any $n-1$ components is finitely generated.

 We call a projection $\pi_{ij}$ from $\C\leqsd\A_1\times\dots\times\A_n$ to $\A_i\times\A_j$ for $i<j$
 \emph{virtually surjective} if $\pi_{ij}(\C)$ is a subdirect product over a finite quotient,
i.e.,
 $\lambda_{ij} := \pi_j(\ker\pi_i\vee\ker\pi_j)$ has finitely many congruence classes in $\A_j$.

As an application of Theorem \ref{thm:mult} we prove the following criterion for finite generation in certain congruence permutable varieties. It certainly applies to groups and rings, and for the former represents a 
small generalization of the finite generation part of \ref{it:G4}, 
due to the fact that one projection is allowed to be simply finitely generated.
 
\begin{cor}  \label{cor:vs2}
 Let $n\geq 2$, and let $\vV$ be a congruence permutable variety such that all congruences
 of finitely generated virtually supernilpotent algebras of class $n-2$ in $\vV$ are finitely generated.
 Let $\A_1,\dots,\A_n\in\vV$ be finitely generated, and let $\C\leqsd \A_1\times\dots\times\A_n$ such that
 $\pi_{12}(\C)$ is finitely generated and $\pi_{ij}$ is virtually surjective for all $1\leq i<j\leq n, (i,j)\neq(1,2)$.
 Then $\C$ is finitely generated.
\end{cor}

\begin{proof}
 We use induction on $n$. The base case for $n=2$ holds by the assumption on $\pi_{12}(\C)$.
 For $n\geq 3$ the other projections $\pi_{ij}(\C)$ for $1\leq i<j\leq n, (i,j)\neq(1,2)$,
 are subdirect products of $\A_i$ and $\A_j$ over some finite quotient and are finitely generated by
 Proposition~\ref{pro:lAfg}.
 For checking the remaining conditions in Theorem~\ref{thm:mult}\ref{it:mfg}
 let $\lambda_{in} := \pi_n(\ker\pi_i\vee\ker\pi_n)$ for $i<n$. 
Since $\pi_{in}$ is virtually surjective,
 $\A_n/\lambda_{in}$ is finite. 
 Since the $(n-1)$-ary commutator of $\bigwedge_{i=1}^{n-1}\lambda_{in}$ with itself is below
 $\gamma_n := [\lambda_{1n},\dots,\lambda_{n-1n}]$ by~\ref{it:hc2},
 the algebra $\A_n/\gamma_n$ is virtually supernilpotent of class at most $n-2$.
 Let $\B$ be the projection of $\C$ on the first $n-1$ coordinates.
 By Theorem~\ref{thm:mult}\ref{it:mfg} it suffices to show that
 $\C/0\times\dots\times 0\times\gamma_n\leqsd\B\times\A_n/\gamma_n$ is finitely generated.
 Since $\B$ is finitely generated by the induction assumption and all congruences of $\A_n/\gamma_n$ are finitely
 generated by the assumption of the corollary, this follows from Proposition~\ref{pro:lAfg}.
\end{proof}

\section{Groups and loops} \label{sec:group}

In this section we gather some observations about the subdirect products of groups and loops against the general background developed in the previous two sections.

We begin with an example of a finitely generated subdirect product of two non-finitely presented groups, which shows that the converse of Proposition \ref{pro:lAfg} does not hold.

\begin{exa} \label{exa:baumslag}
Let $X$ be a finite set ($|X|\geq 2$) and let $F(X)$ denote the free group on $X$.
Inside $F(X)$ pick a free subgroup $H$ of infinite rank, and let $Y$ be a basis for $H$.
Now make two copies $X^\prime,X^{\prime\prime}$ of $X$, with $x\mapsto x^\prime$ and $x\mapsto x^{\prime\prime}$ bijections, and consider the free group $F(X^\prime\cup X^{\prime\prime})$, with the above bijections extending to two embeddings $F(X)\rightarrow F(X^\prime\cup X^{\prime\prime})$. 

Let $R:=\{ y^\prime(y^{\prime\prime})^{-1}\::\: y\in Y\}\subseteq F(X^\prime\cup X^{\prime\prime})$.
The quotient $F(X^\prime\cup X^{\prime\prime})/N(R)$,
where $N(R)$ stands for the normal closure of $R$, is naturally isomorphic to the free product $F(X)\ast_H F(X)$ of two copies of $F(X)$ amalgamated over $H$.
More generally, for any $Z\subseteq Y$ let $R_Z:=\{ y^\prime (y^{\prime\prime})^{-1}\::\: y\in Z\}$, and note that
$F(X^\prime\cup X^{\prime\prime})/N(R_Z)$ is naturally isomorphic to the free product of two copies of $F(X)$ amalgamated over $\langle Z\rangle\leq H\leq F(X)$.
In particular, the normal subgroups $N(R_Z)$ ($Z\subseteq Y$) are pairwise distinct.

Now partition $Y$ into two disjoint infinite subsets $Y=Y_1\cup Y_2$,
and let $R_i:=R_{Y_i}$, $M_i:=N(R_i)\unlhd F(X^\prime\cup X^{\prime\prime})$
for $i=1,2$, and $M:=M_1\cap M_2$.
Note that $M_1M_2=N(R)$.

So the finitely generated group $F(X^\prime\cup X^{\prime\prime})/M$ is
a subdirect product of $F(X^\prime\cup X^{\prime\prime})/M_i$ $(i=1,2)$
over the quotient $F(X^\prime\cup X^{\prime\prime})/N(R)$.
 We claim that neither of the two factor kernels is finitely generated.

The normal subgroup of $F(X^\prime\cup X^{\prime\prime})/M_1$ corresponding to the first factor kernel is $N(R)/M_1$.
Suppose that it is finitely normally generated inside $F(X^\prime\cup X^{\prime\prime})/M_1$.
This would imply the existence of a finite subset $U\subseteq Y_2$ such that $U/M_1$ normally generates $N(R)/M_1$.
But then the normal subgroup of $F(X^\prime\cup X^{\prime\prime})$ generated by $R_{Y_1\cup U}$ would be $N(R)=N(R_Y)$ even though $Y_1\cup U\neq Y$, a contradiction. It follows that the first projection kernel is not finitely generated, and, by symmetry, neither is the second.
\end{exa}

We now move on to subdirect products of more than two factors.
 By Corollary~\ref{cor:vs2} every subdirect product of finitely generated groups with virtual surjective projections
 on all pairs $i<j$ is finitely generated, which, as we mentioned earlier, is already implicit in~\cite{BHM:FPSP}.
In this section we consider whether a weaker finiteness condition for projections to pairs might be sufficient to guarantee finite generation.

 Indeed, for three factors we can prove the following stronger result:

\begin{thm} \label{thm:3groups}
 Let $\C\leqsd \A_1\times\A_2\times\A_3$ be a subdirect product of finitely generated groups $\A_1,\A_2,\A_3$ such that
 the projection onto any two components is a subdirect product over a virtually nilpotent quotient.
 Then $\C$ is finitely generated.
\end{thm}

\begin{proof}
 For $i=1,2,3$, let $N_i$ be the kernel of the projection of $\C$ onto the $i$-th component. Then
\begin{align*}
 N_1 & \leqsd 1\times N_{12} \times N_{13}, \\
 N_2 & \leqsd N_{21}\times 1 \times N_{23}, \\
 N_3 & \leqsd N_{31}\times N_{32} \times 1,
\end{align*}
 for normal subgroups $N_{ij}$ of $\A_j$ ($1\leq i,j\leq 3$).
 Since $\C/N_iN_j$ are finitely generated and virtually nilpotent by assumption, they are
 virtually polycyclic~\cite[5.2.18]{Ro:CTG} and hence finitely presented. 
This group is the quotient of $\pi_{ij}(\C)\cong\C/N_i\cap N_j$ when considered
as a subdirect product of $A_i$ and $A_j$, and hence
$\pi_{ij}(\C)$ is finitely generated by Corollary~\ref{cor:ABfp}.
It now follows from Theorem~\ref{thm:mult}\ref{it:mfg} that $\C$ is finitely generated if and only if 
$\overline{\C}:=\C/K_1\times K_2\times K_3$ is finitely generated,
where $K_i:=[N_{ji},N_{ki}]$, $\{i,j,k\}=\{1,2,3\}$.

Next note that 
$\overline{\C}\leqsd\overline{\A}_1\times\overline{\A}_2\times\overline{\A}_3$,
where $\overline{\A}_i:=\A_i/K_i$.
View $\overline{\C}$ as a subdirect product of $\pi_{12}(\overline{\C})$ and
$\overline{\A}_3$, both of which are finitely generated.
Note that the factor kernel for $\overline{\A}_3$ is
\[
L_{\overline{\A}_3}=\pi_3(\ker\pi_{12})=\pi_3(N_1\cap N_2)/K_3=\pi_3(N_1\cap N_2)/[N_{13},N_{23}].
\]
By Proposition~\ref{pro:lAfg} and Theorem~\ref{thm:mult}\ref{it:mfg} it suffices to show that it is finitely normally generated in $\overline{A}_3$.

 Recall that $\A_3/N_{13}\cong \C/N_1N_3$ and $\A_3/N_{23}\cong \C/N_2N_3$ are virtually nilpotent.
 Then we have a normal subgroup $M$ of finite index in $\A_3$ and $k\in\N$ such that 
 $\gamma_k(M) \leq N_{13}\cap N_{23} \leq M$. Here $\gamma_k(M)$
stands for the $k$-th term of the lower central series of $M$.
 Note that $M$ is finitely generated because it has finite index in the
 finitely generated group $\A_3$. By a result of P. Hall \cite[Theorem 3]{Ha:FCSG}
 every normal subgroup of $M$
 is finitely generated over $[\gamma_k(M),\gamma_k(M)]$. Again using that $|\A_3:M|$ is finite, every normal subgroup
 of $\A_3$ is finitely generated over $[\gamma_k(M),\gamma_k(M)]$, and hence also  over $[N_{13},N_{23}]$.
In particular $\pi_3(N_1\cap N_2)$ is finitely generated over  $[N_{13},N_{23}]$, 
as required. Thus $\overline{\C}$ and $\C$ are finitely generated.
\end{proof}  

\begin{rem}
The assumptions of Theorem \ref{thm:3groups} can be relaxed a tiny bit without altering the conclusion or the proof: it is sufficient to assume that $\pi_{12}(\C)$ is just finitely generated.
\end{rem}

The  following folklore fact, which we give without proof, will  be used to produce abelian normal subgroups
 that are not finitely normally generated.
 
\begin{lem} \label{lem:MX}
 Let $\G$ be a group with subgroup $\H$, and let $\R$ be a commutative ring with $1$.
 Then $\H$ is finitely generated if and only if $\langle -1+h \setsuchthat h\in H\rangle$ is a finitely generated submodule of the
 regular $\R[\G]$-module. 
\end{lem}  

 We now give an example of a non-finitely generated subdirect product of three groups where the projection
 onto any two factors is finitely generated.
 
\begin{exa}  \label{exa:pairsfg1}
 Let $\F$ denote the free group on $\{x,y,z\}$, and let
 \[ \C_0 := \bigl\langle (x,x,x),(y,y,y),(z,z,z),(1,z,z^{-1}),(1,1,z^3) \bigr\rangle. \]
 Clearly $\C_0 \leq \F\times F\times F$ is subdirect. Let $\ZZ$ be the normal subgroup of $\F$ generated by $z$.
 It is straightforward to check that for all $1\leq i<j\leq 3$ we have
 \[ \pi_{ij}(\C_0) = \bigl\langle (x,x),(y,y),(z,z),(1,z) \bigr\rangle = \bigl\{ (a,b)\in\F\times F \setsuchthat a \equiv_\ZZ b \bigr\} =: \B. \]
 For $i\leq 3$, let $N_i$ be the kernel of the projection of $\C_0$ on the $i$-th component. Then
\begin{align*}
 N_1 & \leqsd 1\times \ZZ \times \ZZ, \\
 N_2 & \leqsd \ZZ\times 1 \times \ZZ, \\
 N_3 & \leqsd \ZZ\times \ZZ \times 1.
\end{align*}
 Theorem~\ref{thm:mult}\ref{it:mcom} (or simply computing the commutators $[N_i,N_j]$) yields
 $\ZZ'\times \ZZ'\times \ZZ' \leq \C_0$.

Let $\G:=\langle x,y\rangle\leq \F$, and let 
$\H := \bigl\langle [x^i,y^j] \setsuchthat i,j\in\ZZ\bigr\rangle\leq \F$.
A standard argument via Nielsen reduced sets,
shows that $\H$ is not finitely generated.
Let $\M := [\ZZ,\H]\ZZ^\prime\unlhd \F$.
 We claim that 
\[ \C := \C_0(1\times 1\times \M) \]
 is not finitely generated even though all the projection of $\C$ onto $2$ factors  still equal $\B$ (since $\M\leq \ZZ$)
 and are finitely generated. 

Suppose, aiming for contradiction, that $\C$ is finitely generated.
Consider $\C$ as a subdirect product of $\F$ and $\B$.
Since $\C$ is a semidirect product of $\bigl\langle (x,x,x),(y,y,y),(z,z,z) \bigr\rangle$ and $\ker\pi_\F$,
it follows that the factor kernel $\L:=\pi_\B(\ker\pi_\F)\unlhd\B$ is normally generated by
$\bigl\{(z,z^{-1}),(1,z^3)\bigr\}\cup (1\times M)$.
Since $\F$ is free, and hence finitely presented, and since $\C$ is assumed to be finitely generated,  Proposition~\ref{pro:lBfg}\ref{it:lBfg1} implies that $\L$ must be finitely normally generated.

Next, note that $\ZZ$ is a free group with basis $\{z^g\setsuchthat g\in G\}$.
Let $\overline{\ZZ}:=\ZZ/\ZZ^\prime$, the abelianization of $\ZZ$,
and let $x\mapsto\overline{x}$ be the natural epimorphism $\ZZ\rightarrow \overline{\ZZ}$.
 Consider $\overline{\ZZ}$ as a (right) $\Z[\G]$-module where $\G$ acts on $\overline{\ZZ}$ by conjugation and $\Z$ acts by exponentiation.
 In fact $\overline{\ZZ}$ is the free $\Z[\G]$-module over $\overline{z}$.

Since $\L$ is a finitely generated normal subgroup of $\B$, it follows that $\overline{\L}$ modulo $\ZZ^\prime\times\ZZ^\prime$ is a finitely generated $\Z[\G]$-submodule of $\overline{\ZZ}\times\overline{\ZZ}$.
 Let $\overline{v}_{ij}:=\overline{ \bigl[z,[x^i,y^j]\bigr]}$ for $i,j\in\Z$ be the generators of $\overline{\M}$. 
Then there must exist $k\in\N$ such that 
for all $r,s\in \Z$ we have
\[ (\overline{1},\overline{v}_{r,s}) = (\overline{z},\overline{z}^{-1})^a (\overline{1},\overline{z}^3)^b\ \prod_{|i|,|j|< k}(\overline{1},\overline{v}_{ij})^{c_{ij}} \]
for some $a,b,c_{ij}\in\Z[\G]$.
 Note that $\overline{z}^a=\overline{1}$ implies $(\overline{z}^{-1})^a=\overline{1}$ and hence we have
 $\overline{v}_{r,s} = (\overline{z}^3)^b\ \prod_{|i|,|j|< k}\overline{v}_{ij}^{c_{ij}}$.
 Considering this equation modulo $3$ (i.e., modulo $N := \langle a^3\setsuchthat a\in Z\rangle Z'$) yields
\[ \overline{v}_{r,s} \equiv \prod_{|i|,|j|< k} \overline{v}_{ij}^{c_{ij}} \mod N. \]
 Equivalently, in the regular $\Z_3[\G]$-module, the 
element 
$\overline{v}_{ij}=\overline{z}^{-1}\overline{z}^{[x^i,y^j]}$ corresponds to
$-1+[x^i,y^j]$ and so we have
\[ -1+[x^r,y^s] = \sum_{|i|,|j|< k}\bigl(-1+[x^i,y^j]\bigr) c_{ij}. \]
 Hence the submodule $\langle -1+h \setsuchthat h\in H \rangle$
 is finitely generated, contradicting that the group $\H$ is not finitely generated via Lemma~\ref{lem:MX}. 
\end{exa}

 The next example shows that Theorem~\ref{thm:3groups} does not generalize to an arbitrary number of factors even
 under the stronger assumption that the projection onto any two components is a subdirect product over an abelian
 quotient.

\begin{exa}
 Let $F$ be the free group over $2$ variables, and let
\begin{multline*}
\C_0 := \bigl\langle (a,a,a,a,a,a,a,a),(1,b,1,b,1,b,1,b),\\
(1,1,c,c,1,1,c,c),
(1,1,1,1,d,d,d,d) \setsuchthat
 a\in F,\ b,c,d\in F'\bigr\rangle.
\end{multline*}
 Note that $\C_0$ is isomorphic to $M_\F(1',1',1')$, $\C_0\leqsd\F^8$, and the projection of $\C_0$ onto any two factors
 is $\bigl\{(u,v)\in F\times F\setsuchthat u \equiv v\mod F' \bigr\}$, in particular, a subdirect product over an abelian quotient. 

 By~\cite[Theorem 6]{Ha:FCSG}
 we have a normal subgroup $N$ of $\F$ with $\bigl[[F',F'],F'\bigr] \subseteq N \subseteq [F',F']$ such that $N/\bigl[[F',F'],F'\bigr]$
 is not finitely generated as a normal subgroup. Then $\C_0 (1\times\dots\times 1\times N)$ is not finitely generated
 by an argument that is similar to the one in Example~\ref{exa:pairsfg1}.
\end{exa}

We now briefly turn to loops. Note that they form a congruence permutable variety (when considered as algebras with three binary operations and one constant), and so the results of Section \ref{secFGPK} all apply to them.
However, 
the analogue of Theorem \ref{thm:3groups} does not hold for loops.
In fact, even more strikingly, abelian quotient is not a sufficient condition for finite generation even for two factors, as the following example shows.

\begin{exa}
 Let $\L$ be the free loop over $x$, and let $1$ be the total congruence on $\L$. Note that the abelianization
 $\L/[1,1]$ is isomorphic to the cyclic group $\Z$.
 So we may consider the commutator $[1,1]$ as a subdirect product of $\L\times\L$ over the abelian quotient $\Z$.
 However, $\Z$ is not a finitely presented loop,
 which routinely follows from Evans' solution to the word problem in any finitely presented loop \cite{Ev:LSV}.
 Hence $[1,1]$ is not finitely generated as a congruence of $\L$, and consequently cannot be finitely generated when
 considered as a loop in its own right.
\end{exa}

 Nonetheless, we are able to derive a finite generation result in the case of three factors, and subdirect product
 that is (fully) surjective on pairs as a consequence of Corollary~\ref{cor:vs2}.

\begin{cor} \label{cor:3loops}
 Let $\C\leqsd \A_1\times\A_2\times\A_3$ be a subdirect product of finitely generated loops with surjective projections
 on all pairs. Then $\C$ is finitely generated.
\end{cor}

\begin{proof}
 Any abelian (i.e., (super)nilpotent of class $1$) loop is an abelian group~\cite[p. 295, last paragraph]{SV:CTL}. 
 In particular finitely generated abelian
 loops have finitely generated congruences. Hence the result follows from Corollary~\ref{cor:vs2}.
\end{proof}

 We do not know whether Corollary~\ref{cor:3loops} generalizes to more than three factors, or
 whether surjectivity on pairs can be relaxed to virtual surjectivity.

\section{Modules and $\K$-algebras}
\label{sec:Kalgebras}

The case of modules is interesting because the behaviour of their subdirect products radically depends on the nature
of the underlying ring.
Over Noetherian rings this behaviour is as nice as one could hope for, as reflected in the following two results. 

\begin{lem} \label{lem:mod}
 A ring $\R$ is right Noetherian if and only if every subdirect product of two finitely generated right $\R$-modules is
 finitely generated.
\end{lem}

\begin{proof}
 First let $\A,\B$ be finitely generated modules over a right Noetherian ring $\R$.
 Clearly $\A\times\B$ is finitely generated. Since all submodules of finitely
 generated modules over right Noetherian rings are finitely generated~\cite[Proposition 1.1.12]{HG:ARM},
 in particular all subdirect products of $\A$ and $\B$ are finitely generated.

 Next let $\R$ be a ring that is not right Noetherian. Then it contains a right ideal $N$ that is not finitely
 generated. Consider the right $\R$-module
\[ M := \bigl\{ (a,a+n) \setsuchthat a\in R,\ n\in N \bigr\}. \]
 Then $M$ is a subdirect product of the regular right $\R$-module $R$ with itself. Note that $R$ is $1$-generated
 as a module but that $M\cong R\times N$ is not finitely generated.
\end{proof}

 As an immediate consequence of the previous lemma we obtain the following.

\begin{cor} \label{cor:mod}
 Every subdirect product of finitely many finitely generated right modules over a right Noetherian ring
 is finitely generated.
\end{cor}

Over non-Noetherian rings the behaviour is more in line with the general congruence permutable varieties.
We begin with an example of a finitely presented subdirect product of two non-finitely presented modules.

\begin{exa} \label{exa:Fxymod}
 Let $\R$ be the free non-commutative ring with $1$ over $x,y$, and let $\M$ be the regular right
 module of $\R$ over itself, i.e., $\M$ is the free cyclic right $\R$-module. Consider its submodules
\begin{align*}
 K & := \langle x^iy \setsuchthat i\in\N \rangle, \\
 L & := \langle y^ix \setsuchthat i\in\N \rangle.
\end{align*}  
 Then $K\cap L = 0$, and $M$ is a subdirect product of $\M/K\times \M/L$.
 However the class of $0$ modulo $\lambda_{\M/K}$ is 
\[ K+L/K \cong L/K\cap L \cong L, \]
 hence clearly not finitely generated. Similarly $\lambda_{\M/L}$
 is not finitely generated.
\end{exa}

Perhaps still more surprising are the subdirect products of several factors,
where even surjectivity on all but one components is not sufficient to guarantee finite generation,
as the following example shows.

\begin{exa} \label{exa:pairsfg}
 For $R,M$, and $K$ as in Example~\ref{exa:Fxymod}, let
\[ \C := \{(a_1,\dots,a_n)\in M^n \setsuchthat \sum_{i=1}^n a_i \in K \} \]
 is a subdirect product of $\M^n$ that is surjective on any $n-1$ components. However
\[ \C \cong \M^{n-1}\times \K \]
 is not finitely generated.
\end{exa}

 We close our considerations of modules with the following  explicit connection between a subdirect product of modules in the full direct product and
 the fiber construction, which will be of use subsequently.

\begin{lem} \label{le:modcf}
 Let $\R$ be a ring, let $g\colon\A\to\D, h\colon\B\to\D$ be onto homomorphisms for $\R$-modules $\A,\B,\D$,
 and let $\C := \bigl\{ (a,b)\in A\times B\setsuchthat g(a)=h(b) \bigr\}$ be a fiber
product of $\A$ and $\B$ over $\D$. Then
\[ \A\times\B/\C \cong \D. \] 
\end{lem}

\begin{proof}
 Note that the homomorphism
\[ f\colon \A\times\B \to \D\times\D, (a,b)\mapsto \bigl(g(a),h(b)\bigr), \]
 is onto. Let $\Delta := \{(d,d)\setsuchthat d\in D\}$ be a submodule of $\D^2$. Then $C = f^{-1}(\Delta)$
 and the Homomorphism Theorem yields
\[ \A\times\B/\C \cong \D\times\D/\Delta \cong \D. \]
\end{proof}

We now turn our attention to rings, and begin with an example of a finitely presented subdirect product of two non-finitely presented factors.

\begin{exa}
\label{exa:fprings}
Let $\F=\Z\langle x,y\rangle$ be the free ring with $1$ over two non-commuting variables. Let $\I$ and $\J$ be the ideals of $\F$ generated by the sets 
\[
\{xy^ix\setsuchthat i\in\N\}\cup\{ x^2y^2,y^2x^2,yxy\}
\text{ and }
\{yx^iy\setsuchthat i\in\N\}\cup\{ y^2x^2,x^2y^2,xyx\}
\]
respectively. 
It is clear that neither of them is finitely generated. We will prove that the intersection $I\cap J$ is finitely generated.
Indeed, it is easy to see that the only monomials that do not belong to $I\cap J$ are
$x^i$, $y^i$, $xy^i$, $y^i x$, $x^i y$, $yx^i$ (for $i\in\N$),
and $xy^ix$, $yx^iy$ (for $i\geq 2$).
It then follows that $I\cap J$ is generated by the finite set
$\{ x^2y^2,y^2x^2,xyx,yxy\}$, as required.
\end{exa}

Next we prove the full analogue of the property \ref{it:G4} for groups. For this we work in the more general setting of $\K$-algebras.

For $\K$ a commutative ring with $1$, a \emph{$\K$-algebra} is a structure $\A$ which is simultaneously a ring
(not necessarily commutative or with $1$) and a $\K$-module, such that the ring multiplication is $\K$-bilinear. Clearly, the notion of $\K$-algebras encompasses those of rings ($\K=\Z$) and classical algebras ($\K$ a field).

For $\K$-algebras with $\K$ Noetherian we will be able to prove the full analogue of
 the result for groups \ref{it:G4} by Bridson et al.~\cite{BHM:FPSP}.
In fact, we will be able to relax the condition of virtual surjectivity on pairs to that of \emph{finite co-rank},
defined as follows.
If $N$ is a submodule of $M$ and the quotient $M/N$ is finitely generated we say that $N$ has \emph{finite co-rank}
 in $M$.
 Finite co-rank plays the role for $\K$-algebra analogous to that of finite index for groups as demonstrated in~\cite{MR:PSS}.
 There it was proved, for a $\K$-algebra $\A$ over a Noetherian commutative ring $\K$ with $1$ and a subalgebra $\B$ of finite co-rank, that $\A$ is finitely generated (resp. finitely presented) if and only if $\B$ is finitely generated (resp. finitely presented).

\begin{lem}
\label{lem:Kalgebra}
Let $\K$ be a commutative Noetherian ring with $1$, $n\geq 2$, and $\C\leqsd\A_1\times\dots\times\A_n$ a subdirect
 product of finitely generated $\K$-algebras $\A_1,\dots,\A_n$. Then
$\C$ has finite co-rank in $\A_1\times\dots\times\A_n$
if and only if
$\pi_{ij}(\C)$  has finite co-rank in $\A_i\times\A_j$ for all $1\leq i<j\leq n$.
\end{lem}

\begin{proof}
For $n=2$ there is nothing to prove, so we assume $n\geq 3$.
Also the 
forward part is obvious, and we only prove 
the backward direction.

For $i\leq n$, let the ideal $I_i$ denote the kernel of the projection $\pi_i\colon C\to A_i$. Then
\[ I_i \leqsd I_{i1}\times\dots\times I_{i,i-1}\times 0 \times I_{i,i-1}\times\dots\times I_{in} \] 
 for ideals $I_{ij}$ of $\A_j$. For $1\leq i<j\leq n$ the $\K$-module $(A_i\times A_j)/\pi_{ij}(C)$ is isomorphic
 to $A_j/I_{ij}$ by Lemma~\ref{le:modcf}. Hence $A_j/I_{ij}$ has finite rank by the assumption of the theorem.
 We claim that also
\begin{equation} \label{eq:fcr}
 A_n/I_{1n}\cdots I_{n-1,n} \text{ has finite rank}.
\end{equation}
 For proving this, let $J_n:= I_{1n}\cdots I_{n-1,n}$ and  $L:= I_{1n}\cap\dots\cap I_{n-1,n}$. Then $J_n\leq L$ and 
 $A_n/L$ has finite rank as a subdirect product of modules of finite rank by Corollary~\ref{cor:mod}. 
 Since $\A_n/J_n$ is a finitely generated $\K$-algebra and $\A_n/L$ is a finitely generated
 $\K$-module, $L/J_n$ is generated by a finite set $X$ as a $\K$-algebra by~\cite[Theorem 1.5]{MR:PSS}.
 Note that the ideal product $L^{n-1}\subseteq J_n$, that is, $L/J_n$ is nilpotent. It follows that every
 element of $L/J_n$ is a $\K$-linear combination of monomials over $X$ of degree at most $n-2$.
 Hence the finite set $\bigcup_{\ell=1}^{n-2} X^\ell$ generates $L/J_n$ as $\K$-module and~\eqref{eq:fcr} follows.

 Similarly for any $j\leq n$ the ideal $J_j := \prod_{i\neq j} I_{ij}$ has finite co-rank in $\A_j$.
 Now
\[ J_1\times\dots\times J_n \leq \C \leq \A_1\times\dots\times\A_n \]
 by Theorem~\ref{thm:mult}\ref{it:mcom} (or straightforward multiplication of the ideals $I_j$).
 Since $J_1\times\dots\times J_n$ has finite co-rank in $\A_1\times\dots\times\A_n$, so has $\C$.
\end{proof}

\begin{thm} \label{thm:Kalgebra}
 Let $\K$ be a commutative Noetherian ring with $1$, $n\geq 2$, and $\C\leqsd\A_1\times\dots\times\A_n$ a subdirect
 product of $\K$-algebras $\A_1,\dots,\A_n$ such that
 $\pi_{ij}(\C)$  has finite co-rank in $\A_i\times\A_j$ for all $1\leq i<j\leq n$. Then
 $\C$ is finitely generated (finitely presented) if and only if all $\A_1,\dots,\A_n$ are finitely generated (finitely presented).
\end{thm}  

\begin{proof}
Note that the conditions imply that $\C$ has finite co-rank in $\A_1\times\dots\times\A_n$ by Lemma \ref{lem:Kalgebra}.

($\Rightarrow$)
For finite generation this is obvious.
For finite presentability, we have that $\A_1\times\dots\times\A_n$ is finitely presented by
\cite[Theorem 1.1]{MR:PSS}. Then all $\A_i$ are finitely presented by \cite[Theorem 3.2]{MR:FPD}.

($\Leftarrow$)
 If $\A_1,\dots,\A_n$ are finitely generated (resp. finitely presented), then so is their direct product
 $\A_1\times\dots\times\A_n$ by~\cite[Corollary 2.4 (Corollary 3.7)]{MR:FPD}. It then follows that $\C$ is finitely
 generated (resp. finitely presented) by~\cite[Theorem 1.5 (Theorem 1.1)]{MR:PSS}.
\end{proof}  

Attempting to go beyond finite co-rank, one may ask whether the analogue of Theorem \ref{thm:3groups} holds for
$\K$-algebras. As mentioned earlier, for a $\K$-algebra $\R$ to be nilpotent means that
the ideal product
$\R^n=0$ for some $n>0$.
Furthermore if $\R$ is a finitely generated nilpotent $\K$-algebra, then it is in fact finitely generated as a
$\K$-module.
Combining this with Lemma \ref{le:modcf} implies that if $\C\leqsd \A_1\times \A_2$ with a nilpotent quotient, then $\C$ has a finite co-rank in $\A_1\times \A_2$.
Thus the analogue of Theorem \ref{thm:3groups}, and its generalization to any number of factors, both hold for $\K$-algebras as an immediate consequence of Theorem \ref{thm:Kalgebra}.

By way of contrast, Example \ref{exa:pairsfg1} carries over to rings in a straightforward fashion to show that just finite generation of projections on pairs is not sufficient for finite generation of the subdirect product.

\begin{exa}
\label{exa:3rings}
Let $\F:=\Z\langle x,y,z\rangle$ denote the free ring in three non-commuting variables with 1, and let
\[
\C_0:=\bigl\langle (x,x,x),(y,y,y),(z,z,z),(0,z,-z),(0,0,3z)\bigr\rangle \leqsd \F\times\F\times\F.
\]
Let $\ZZ$ be the ideal of $\F$ generated by $z$, and let
\[
B:=\bigl\{ (a,b)\in \F\times F\setsuchthat a-b\in Z\bigr\}=\bigl\langle (x,x),(y,y),(z,z),(0,z)\bigr\rangle .
\]
 It is clear that
 $\pi_{ij}(\C_0) =  \B $ for all $1\leq i<j\leq 3$.
 For $i=1,2,3$, let $N_i$ be the kernel of the projection of $\C_0$ on the $i$-th component. Then
\begin{align*}
 N_1 & \leqsd 0\times \ZZ \times \ZZ, \\
 N_2 & \leqsd \ZZ\times 0 \times \ZZ, \\
 N_3 & \leqsd \ZZ\times \ZZ \times 0.
\end{align*}
 Computing the products $N_iN_j$ yields $\ZZ^2\times \ZZ^2\times \ZZ^2 \leq \C_0$.

 Let $\M$ be the ideal of $\F$ generated by
 $\{ xy^iz\setsuchthat i=1,2,\dots\}\cup \ZZ^2$, and let
 \[ \C := \C_0+0\times 0\times \M\leqsd \F\times\F\times\F. \]
We claim that $\C$ is not finitely generated even though all the projections of $\C$ onto $2$ factors still
equal $\B$ (since $\M\leq \ZZ$) and are finitely generated. 

Suppose, aiming for contradiction, that $\C$ is finitely generated.
Consider $\C$ as
a subdirect product of $\F$ and $\B$.
A similar (and easier) argument to that in Example \ref{exa:pairsfg1} shows that
the
factor kernel $\L:=\pi_\B(\ker\pi_\F)$ is generated by $\{(z,-z),(0,3z)\}\cup (0\times M)$ as an ideal of $\B$.
Since $\F$ is free, and hence finitely presented, and since $\C$ is assumed to be finitely generated,  Proposition~\ref{pro:lBfg}\ref{it:lBfg1} implies that $\L$ must finitely  generated as an ideal of $\B$.

Consider the quotient $\overline{\ZZ}:=\ZZ/\ZZ^2$, and regard it as a (free) $\G$-bimodule, where
$\G:=\langle x,y\rangle\leq \F$.
Since $\L$ is a finitely generated ideal of $\B$ and since $\ZZ^2\times\ZZ^2\leq\L$, it follows that $\overline{\L}$
 is a finitely generated bimodule of $\overline{\ZZ}\times\overline{\ZZ}$.
Hence there exists $k\in \N$ such that for every $r\in\N$ we have
\[
(0,xy^rz)=a(z,-z)a^\prime+b(0,3z)b^\prime +\sum_{i=1}^{k} c_i(0,xy^iz)c_i^\prime
\]
for some $a,a^\prime,b,b^\prime,c_i,c_i^\prime\in\G$.
Noting that $a=0$ or $a^\prime=0$ on account of the first components, and then reducing modulo $3$, we obtain
\[
xy^rz\equiv\sum_{i=1}^{k}c_i(0,xy^iz)c_i^\prime.
\]
This would imply that the ideal generated by 
$\{ xy^iz\setsuchthat i=1,2,\dots\}\cup\ZZ^2$
in the ring $\Z_3\langle x,y,z\rangle$ is finitely generated, which is clearly not the case.
\end{exa}

\section{Beyond congruence permutable: monoids and lattices}
\label{sec:smgplat}

In this section we present some examples of subdirect products of monoids and lattices which serve to show that our main results do  not readily generalise beyond congruence permutable or modular varieties, as appropriate.

We begin with monoids, and give an example of a fiber product of two copies of the free monoid of rank 1 over a finite quotient that is not finitely generated.

\begin{exa}
\label{exmon1}
Let $\F_{1}$ be the free monoid on a single generator $a$ with identity $1$, 
and let $\T$ be the two element monoid $\{0,1\}$ under the standard multiplication.
Define a homomorphism $\phi:\F_{1}\rightarrow\{0,1\}$ by $a\mapsto 0$.
The fiber product $\S:=\ker\phi$ is not finitely generated.
Indeed, clearly $S=\bigl\{(1,1)\bigr\}\cup \bigl\{ (a^i,a^j)\setsuchthat i,j>0\bigr\}$.
Each of the elements $(a^i,a)$ ($i\in\N$) is indecomposable, i.e., not equal 
to a product of two other non-identity elements of $\S$, and hence must belong to every generating set of $\S$.
\end{exa}

This example indicates that it is hard for a fiber product of two free monoids over a non-trivial monoid to be finitely generated. We ask:

\begin{prob}
\label{monsdfg}
 Find necessary and sufficient conditions for a fiber product of finitely generated monoids over a finite monoid to 
 be finitely generated. More specifically, is
 it decidable whether a fiber product of two finitely generated free monoids over a finite quotient is finitely
 generated?
\end{prob}

Our next example shows that,
for monoids, finite presentability of a subdirect product $\C\leqsd \A\times \B$
and of the associated quotient are not sufficient to ensure finite generation of the factor kernels $\lambda_\A$ and $\lambda_\B$.
It therefore shows that Proposition \ref{pro:lBfg}\ref{it:lBfg2} does not generalize to monoids.

\begin{exa}
\label{ex:monlambda}
Let $\F$ be the free monoid on two generators $\{x,y\}$, and let
$\sigma,\tau$ be its congruences generated by
\begin{align*}
& \bigl\{ (xy^ix,xyx),(x^2y^2,x^2y),(y^2x^2,yx^2)\setsuchthat i\in\N\bigr\}, \\
& \bigl\{ (yx^iy,yxy),(y^2x^2,y^2x),(x^2y^2,xy^2)\setsuchthat i\in\N\bigr\},
\end{align*}
respectively.
We claim that the following hold:
\begin{enumerate}[label=(\alph*), widest=(iii), leftmargin=10mm,topsep=0mm,partopsep=0mm,itemsep=1mm]
\item
\label{it:monF1}
Neither $\sigma$ nor $\tau$ are finitely generated.
\item
\label{it:monF2}
$\sigma\cap\tau=0_F$, so that $\F\leqsd \F/\sigma\times\F/\tau$.
\item
\label{it:monF3}
$\sigma\vee\tau$ is finitely generated, so that the corresponding quotient is finitely presented, and hence the factor kernels are finitely generated.
\end{enumerate}

For~\ref{it:monF1}, $\sigma$ is not finitely generated because the only left- or right hand side of a generator that is a subword of $xy^ix$ is $xy^i x$ itself, and the proof for $\tau$ is dual.
Statement~\ref{it:monF2} follows by observing that applications of generators of $\sigma$ do not change powers of $x$, while those of $\tau$ do not change powers of $y$.
Finally, we prove~\ref{it:monF3} by showing that $\sigma\join \tau$ is generated by
\[
\{ (xy^2x,xyx),(yx^2y,yxy),(x^2y^2,x^2y),(y^2x^2,yx^2),
(y^2x^2,y^2x),(x^2y^2,xy^2)\}.
\]
Denoting by $\rho$ the congruence generated by this set, and supposing inductively that $(xy^ix,xyx),(yx^iy,yxy)\in\rho$ for some $i\geq 2$, we have
\begin{align*}
xy^{i+1}x&=xy^{i-1}y^2x\equiv_\rho xy^{i-1} y^2x^2\equiv_\rho
xy^{i-1}yx^2=xy^{i-2}y^2x^2\\
&\equiv_\rho xy^{i-2}y^2x=xy^ix\equiv_\rho
xyx,
\end{align*}
and, dually, $(yx^{i+1}y,yxy)\in\rho$.
So all generators of $\sigma$ and $\tau$ belong to $\rho$, and the assertion follows.
\end{exa}

Our final monoid example shows that in the case of several factors surjectivity on pairs is not sufficient to ensure finite generation.

\begin{exa}
\label{exmon3}
Let $\N_0=\{0,1,2,\dots\}$ be the free monogenic (cyclic) monoid.
Let $\S\leqsd \N_0\times\N_0\times\N_0$ be given by the generators $(1,0,3)$ and all its permutations, as well as $(0,2,n)$ for $n=7,8,\dots$.
Clearly, $\S$ is surjective on pairs, due to the first group of generators.
We claim that all generators in the second group are indecomposable.
Suppose not, and write such a generator as a sum of generators:
\[
(0,2,n)=g_1+g_2+\dots+g_k
\]
where $k>1$. No $g_i$ can equal $(0,2,m)$, because then the remaining generators would need to sum to $(0,0,n-m)\not\in S$.
Every generator $(x,y,z)$ of the first type, except for $(0,1,3)$, has $x>0$ or $y>2$.
Hence no $g_i$ equals any of these.
Therefore we must have all $g_i=(0,1,3)$, which implies $(0,2,n)=(0,k,3k)$,
and hence $k=2$ and $n=6<7$, a contradiction.
\end{exa}

We now turn to lattices.
For several factors, the Baker--Pixley Theorem \cite[Lemma IV.10.4]{BS:CUA} asserts that a subdirect product of
  lattices that is surjective on pairs must in fact be the full direct product. Further such a (sub)direct product
  of finitely generated lattices is finitely generated~\cite[Corollary 2.6]{MR:FPD}.
 However a subdirect product of finitely generated lattices that is virtually surjective on pairs is not
 necessarily finitely generated. This can be observed from the following
 example of a fiber product of two finitely generated lattices over a finite quotient that is not finitely generated.

\begin{exa} \label{exa:lat}
 The \emph{diamond} $M_3$ is the lattice with three atoms $\{a,b,c\}$ any two of which join to $1$;
 see Figure \ref{fig:M3}.
 Let $F$ be the free lattice over three generators $\{x,y,z\}$, and define an epimorphism
\[ \phi\colon F\to M_3\ \text{ by }\ x \mapsto a, y \mapsto b, z\mapsto c. \]
 Then $\C := \ker\phi$ is a fiber product of $F$ with itself over the the quotient
 $M_3$. To prove that $\C$ is not finitely generated, we first establish some infinite ascending chain
 of elements $x_n$ in $\phi^{-1}(a)$ and show that no finite subset of $\C$ can generate all pairs $(x,x_n)$ for
 $n\in\N$.

 \begin{figure}
\begin{center}
\begin{tikzpicture}
\draw[fill] (4,0) circle [radius=0.8mm] node [below] {$0$};
\draw[fill] (4,2) circle [radius=0.8mm] node [above] {$1$};
\draw[fill] (4,1) circle [radius=0.8mm] node [right] {$b$};
\draw[fill] (2.5,1) circle [radius=0.8mm] node [left] {$a$};
\draw[fill] (5.5,1) circle [radius=0.8mm] node [right] {$c$};
   \draw (4,0)--(4,2);
   \draw (4,0)--(2.5,1)--(4,2);
   \draw (4,0)--(5.5,1)--(4,2);
\end{tikzpicture}
\end{center}
\caption{\textsf{The diamond $M_3$.}}
\label{fig:M3}
\end{figure}

 For $n\in\N$ define
 \[ x_0 := x, \quad y_0 := y,\quad  z_0 := z, \]
 \[ x_{n+1} := x\vee(y_n\wedge z_n), \quad y_{n+1} := y\vee(x_n\wedge z_n), \quad z_{n+1} := z\vee(x_n\wedge y_n). \]
 We claim that for all $n,m\in\N$
\begin{enumerate}[label=\textup{(\alph*)}, widest=(b), leftmargin=10mm]
\item \label{it:hxa}
 $\phi(x_n) = a$, 
\item \label{it:inc}
 $x_n$ is incomparable with $y_m$ and $z_m$, 
\item \label{it:xac}
   $x_n < x_{n+1}$, 
   \item \label{it:xy}
  $x_{n+1}\wedge y_{n+1} \not\leq x_m$ if $m\leq n$, 
\item \label{it:pxn}
 for each $p\in\phi^{-1}(\{0,a\})$ there exists $k\in\N$ such that $p\leq x_k$,  
\end{enumerate}
and analogous statements for the sequences $y_n,z_n$ and $b,c$, respectively.

 \ref{it:hxa} is immediate by induction on $n$.

 \ref{it:inc} Since $a,b,c$ are pairwise incomparable in $M_3$,
 so are $x_n\in\phi^{-1}(a)$ and $y_m\in\phi^{-1}(b), z_m\in\phi^{-1}(c)$. 

 \ref{it:xac},\ref{it:xy} are proved together by induction on $n$. For $n=0$
\[ x_0 = x < x\vee(y\wedge z) = x_1 \]
 is well known~\cite[Figure 123]{Gr:LTF}.
 Further, if $x_1\wedge y_1 \leq x_0$, then either $x_1\leq x$ or $y_1\leq x$ since generators in a free lattice
 are meet prime~\cite[Corollary 1.5]{FJN:FL}. The former contradicts $x<x_1$, the latter~\ref{it:inc}.
 Hence $x_1\wedge y_1 \not\leq x_0$ and the base case is complete.

 For the induction step for~\ref{it:xac} $x_n \leq x_{n+1}$ is straightforward. Seeking a contradiction
 suppose $x_{n+1} \leq x_n$. Then in particular $y_n\wedge z_n\leq x_n = x\vee (y_{n-1}\wedge z_{n-1})$.
 Whitman's condition for comparing meets and joins in free lattices~\cite[Theorem 1.8]{FJN:FL}
 implies that one of the following holds: $y_n\wedge z_n\leq x$, $y_n\wedge z_n\leq y_{n-1}\wedge z_{n-1}$, $y_n\leq x_n$
 or $z_n\leq x_n$.
 The first option yields $y_n\leq x$ or $z_n\leq x$ contradicting~\ref{it:inc}, the second yields
 $y_n\wedge z_n\leq y_{n-1}$ contradicting the induction hypothesis for~\ref{it:xy}, the final two
 contradict~\ref{it:inc}. Thus $x_{n+1} \not\leq x_n$ and the induction step for~\ref{it:xac} is proved.

 For the induction step for~\ref{it:xy} we use a secondary induction on $m$. The base case
 $x_{n+1}\wedge y_{n+1} \not\leq x_0$ holds because otherwise $x_{n+1}\leq x$ (contradicting $x_{n+1}\not\leq x_n$)
 or $y_{n+1}\leq x$ (contradicting~\ref{it:inc}).
 For the induction step, seeking a contradiction we suppose $x_{n+1}\wedge y_{n+1}\leq x_{m+1} = x\vee (y_m\wedge z_m)$
 for $m<n$. By Whitman's condition we get
\begin{itemize}
\item $x_{n+1}\wedge y_{n+1}\leq x$, which contradicts the base case,
\item $x_{n+1}\wedge y_{n+1}\leq y_m\wedge z_m$, which contradicts the induction hypothesis,
\item $x_{n+1}\leq x_{m+1}$, which contradicts $x_{n+1}\not\leq x_n$, or
\item $y_{n+1}\leq x_{m+1}$, which contradicts \ref{it:inc}.
\end{itemize}
 This completes the induction on $m$ and on $n$. Thus~\ref{it:xac} and~\ref{it:xy} are proved.
 
 \ref{it:pxn} is proved by induction on the complexity of $p$. If $p\in\phi^{-1}(\{0,a\})$ is one of the
 generators $x,y,z$, then $p=x\leq x_0$ by~\ref{it:hxa}.

 Next assume $p = p_1\vee p_2$ for $p_1,p_2\in F$. Then $p_1,p_2$ are in $\phi^{-1}(\{0,a\})$. By the induction
 hypothesis and~\ref{it:xac} we have $k\in\N$ such that $p_1,p_2$ are both below $x_k$. So $p\leq x_k$.

 Finally assume $p = p_1\wedge p_2$ for $p_1,p_2\in F$. If $\phi(p_1) \leq a$ or $\phi(p_2) \leq a$, the assertion
 follows from the induction hypothesis. Else we may assume without loss of generality that $\phi(p_1) = b$ and
 $\phi(p_2) = c$.  By the induction hypothesis and~\ref{it:xac} we have $k\in\N$ such that $p_1\leq y_k$ and
 $p_2\leq z_k$. Then $p \leq y_k\wedge z_k \leq x\vee(y_k\wedge z_k) = x_{k+1}$. This completes the proof of~\ref{it:pxn}.

 Let $X$ be a finite subset of $C$. We show that $\langle X\rangle \neq C$.
 By~\ref{it:pxn} and~\ref{it:xac} there exists $k\in\N$ such that for every $(p,q)\in X$ with $p\leq x$
 we have $q\leq x_k$. We claim that
\begin{equation} \label{eq:qxn}
 \forall (p,q)\in\langle X\rangle\colon p\leq x \Rightarrow q\leq x_k.
\end{equation}
 The proof is by induction on the complexity of $(p,q)$ as a term over $X$. Let $(p,q)\in\langle X\rangle$ and
 $p\leq x$.
 For $(p,q)\in X$ the statement follows from the definition of $k$.

 Next assume $(p,q) = (p_1,q_1)\vee(p_2,q_2)$ for $(p_i,q_i)\in\langle X\rangle$. Then $p_1, p_2\leq x$
 and $q_1,q_2\leq x_k$ by the induction hypothesis. Thus  $q=q_1\vee q_2\leq x_k$.

 Finally assume $(p,q) = (p_1,q_1)\wedge(p_2,q_2)$ for $(p_i,q_i)\in\langle X\rangle$. Since generators in a free
 lattice are meet prime, we may assume $p_1\leq x$ without loss of generality. Then $q_1\leq x_k$ by induction
 assumption and consequently $q=q_1\wedge q_2\leq x_k$. This completes the proof of~\eqref{eq:qxn}.

 Together with~\ref{it:hxa} and~\ref{it:xac} it now follows that $(x,x_{k+1}) \in C \setminus\langle X\rangle$.
 Thus $C$ is not finitely generated.
\end{exa}
 
 In analogy to Problem~\ref{monsdfg} we ask:

\begin{prob}
 For $\phi:\F_{r}\rightarrow\D$ a homomorphism from the free lattice $\F_r$ of rank $r$ onto a finite lattice $\D$,
 find necessary and sufficient conditions for the fiber product $\ker\phi \leqsd \F_{r}\times \F_{r}$ 
to be finitely generated.
\end{prob}


\def\cprime{$'$}

\end{document}